\documentclass[12pt]{amsart}
\usepackage[dvips]{graphicx}
\usepackage{amssymb,amsmath,amsthm}

\def\bold{\bf}

\def\0b{{\bold 0}}
\usepackage{bm}
\bmdefine{\Bzero}{0}
\bmdefine{\Bone}{1}
\def\Bone{{\bf 1}}

\def\ZZ{{\mathbb Z}}

\def\CC{{\mathbb C}}
\newtheorem{Theorem}{Theorem}[section]
\newtheorem{Lemma}[Theorem]{Lemma}

\newtheorem{Remark}[Theorem]{Remark}

\newtheorem{Conjecture}[Theorem]{Conjecture}

\begin{document}
\title{Smooth Fano polytopes whose Ehrhart polynomial has a root with large real part}
\author{Hidefumi Ohsugi and Kazuki Shibata}
\thanks{
The authors are grateful to Takayuki Hibi, Akihiro Higashitani and
Tetsushi Matsui for useful discussions.
This research was supported by 
JST CREST
}
\date{}
\address{Hidefumi Ohsugi,
Department of Mathematics,
College of Science,
Rikkyo University,
Toshima-ku, Tokyo 171-8501, Japan.
e-mail address:{\tt ohsugi@rikkyo.ac.jp}
}
\address{Kazuki Shibata,
Department of Mathematics,
College of Science,
Rikkyo University,
Toshima-ku, Tokyo 171-8501, Japan.
e-mail address:{\tt 10lc002r@rikkyo.ac.jp}
}

\begin{abstract}
The symmetric edge polytopes of odd cycles
(del Pezzo polytopes)
are known as smooth Fano polytopes.
In this paper, we show that if the length of the cycle is 127,
then the Ehrhart polynomial has a root whose real part is greater than the dimension.
As a result,
we have a smooth Fano polytope that is
a counterexample to the two conjectures on 
the roots of Ehrhart polynomials.
\end{abstract}

\maketitle

\section*{Introduction}

Let $d \geq 3$ be an integer and $A_d$,
the $(d+1) \times (2d +1)$ matrix
$$
A_d = 
\left(
\begin{array}{c|cccc|cccc}
0 & 1 &  &  & -1 & -1   &  &  & 1 \\
0 & -1 & \ddots  &  &  & 1 & \ddots   &  &  \\
\vdots         &     & \ddots & 1 &  &    & \ddots & -1 &  \\
0 &   &  & -1 & 1 &  &  &  1 & -1 \\
\hline
1 & 1 & \cdots & 1 & 1 & 1 & \cdots & 1 & 1 \\
\end{array}
\right)
.$$
In the present paper, we study
the convex hull ${\rm Conv}(A_d)$ of $A_d$.
The matrix $A_d$ is the {\em centrally symmetric configuration} \cite{central}
and ${\rm Conv}(A_d)$ is called the {\em symmetric edge polytope} of the cycle
of length $d$.
From the results in \cite{Higa1, five},
we have

\medskip

\noindent
{\bf Proposition.}
The polytope $ {\rm Conv}(A_d)$ is a Gorenstein Fano polytope
 (reflexive polytopes)
of dimension $d - 1$.
In addition, $ {\rm Conv}(A_d)$ is a smooth Fano polytope
if and only if $d$ is odd.

\medskip

Here, we first construct  
the reduced Gr\"obner basis ${\mathcal G}$ of $I_{A_d}$.
Next, using ${\mathcal G}$, we compute the Ehrhart polynomial 
and the $h$-vector
of ${\rm Conv}(A_d)$.
Finally, we study the roots of 
the Ehrhart polynomial when $d$ is odd.
We show that the Ehrhart polynomial of ${\rm Conv}(A_{127})$ has a root 
whose real part is greater than $\dim ( {\rm Conv}(A_{127})  )$.
This is a counterexample to the conjectures given in \cite{BDDPS2005, five}.

\section{Gr\"obner bases of toric ideals}

Let ${\mathcal R}_d 
= K[t_1, t_1^{-1}, \ldots, t_{d}, t_{d}^{-1},s]$ be
the Laurent polynomial ring over a field $K$
and
let $K[z,X,Y]=K[z,x_1,\ldots, x_d , y_1 , \ldots , y_d]$ be the polynomial ring
over $K$.
We define the 
ring homomorphism $\pi :K[z,X,Y] \rightarrow {\mathcal R}_d$
by setting $\pi(x_i) = t_i t_{i+1}^{-1} s$, $\pi(y_i) = t_i^{-1} t_{i+1} s$ 
for $1 \leq i \leq d$ (here we set $t_{d+1}=t_1$) and $\pi(z) = s$. 
The {\em toric ideal} $I_{A_d}$ is $\ker (\pi)$.
Let $<$ be the reverse lexicographic order on $K[z,X,Y]$ with the ordering
$
z < y_d < x_d < \cdots < y_1< x_1
.$
For $d \geq 3$, let $[d] = \{1,\ldots,d\}$ and $k = \lceil \frac{d}{2} \rceil$.

\begin{Theorem}
\label{oddGB}
The reduced Gr\"obner basis of $I_{A_d}$ with respect to $<$
consists of
\begin{eqnarray}
x_i y_i - z^2 & (1 \leq i \leq d) \\
\prod_{l=1}^{k} x_{i_l} - z \prod_{l=1}^{k-1} y_{j_l}
&
\left(
[d] = \{ i_{1} ,\ldots , i_{k}  \} \cup 
\{ j_{1} , \ldots , j_{k-{1}} \}
\right)\\
\prod_{l=1}^{k} y_{j_{l}} - z \prod_{l=1}^{k-1} x_{i_{l}}
&
\left(
[d] = \{ i_{1} ,\ldots , i_{k-{1}}  \} \cup 
\{ j_{1} , \ldots , j_{k} \}
\right)
\end{eqnarray}
if $d$ is odd and
\begin{eqnarray}
x_i y_i - z^2 & (1 \leq i \leq d) \\
\prod_{l=1}^{k} x_{i_l} - y_d \prod_{l=1}^{k-1} y_{j_l}
&
\left(
[d-1] = \{ i_{1} ,\ldots , i_{k}  \} \cup 
\{ j_{1} , \ldots , j_{k-{1}} \}
\right)\\
\prod_{l=1}^{k} y_{j_{l}} - x_d \prod_{l=1}^{k-1} x_{i_{l}}
&
\left(
[d-1] = \{ i_{1} ,\ldots , i_{k-{1}}  \} \cup 
\{ j_{1} , \ldots , j_{k} \}
\right)
\end{eqnarray}
if $d$ is even.
The initial monomial of each binomial is the first monomial.
\end{Theorem}

\begin{proof}
Let ${\mathcal G}$ be the set of all binomials above.
It is easy to see that ${\mathcal G} \subset I_{A_d}$ and that
the initial monomial of each binomial in ${\mathcal G}$ is the first monomial.
Let ${\rm in} ({\mathcal G}) =
\langle {\rm in}_<(g) \ | \ g \in {\mathcal G} \rangle$.
Suppose that $d$ is odd and that
${\mathcal G}$ is not a Gr\"obner basis of $I_{A_d}$.
Then, there exists an irreducible binomial $(0 \neq) \ f =  u - v \in I_{A_d}$
such that
neither $u$ nor $v$ belongs to ${\rm in} ({\mathcal G})$.
Let
$
u = z^{\alpha}  \prod_{l=1}^{m} x^{p_{l}}_{i_{l}}   \prod_{l=1}^{n} y^{q_{l}}_{j_{l}}$
and
$v = z^{\alpha'} \prod_{l=1}^{m'} x^{p'_l}_{i'_l}   \prod_{l=1}^{n'} y^{q'_l}_{j'_l}  
$,
where
$0 < p_l , q_l, p'_l , q'_l \in \ZZ$
for all $l$
and
$
{\mathcal I} = \{ i_1 , \ldots , i_{m} \}$, 
${\mathcal J} = \{ j_1 , \ldots , j_{n} \}$, 
${\mathcal I}' = \{ i'_1 , \ldots , i'_{m'} \}$,
${\mathcal J}' = \{ j'_1 , \ldots , j'_{n'} \}$
are subsets of $[d]$
with the cardinality $m$, $n$, $m'$, and $n$, respectively.
Since neither $u$ nor $v$ is divided by $x_i y_i$,
we have
${\mathcal I} \cap {\mathcal J} = {\mathcal I}' \cap {\mathcal J}' = \emptyset.$
In addition,  since neither $u$ nor $v$ is divided by
the initial monomials of binomials (2) and (3),
it follows that
$m,n, m',n' \leq k-1$.
Moreover, since $f$ is irreducible, 
we have 
${\mathcal I} \cap {\mathcal I}' = {\mathcal J} \cap {\mathcal J}' = \emptyset$.
Let $p = \sum_{l=1}^{m} p_{l}$, $q=\sum_{l=1}^{n} q_{l}$, 
$p'=\sum_{l=1}^{m'} p'_l$, $q'=\sum_{l=1}^{n'} q'_l$.
Then,
$
\pi(u) = s^{\alpha + p + q}
\prod_{l=1}^{m} (t_{i_l} t^{-1}_{i_l +1})^{p_l}
\prod_{l=1}^{n} (t^{-1}_{j_l} t_{j_l +1})^{q_l}
$,
$
\pi(v) = s^{\alpha' + p' + q'}
\prod_{l=1}^{m'} (t_{i'_l} t^{-1}_{i'_l +1})^{p'_l}
\prod_{l=1}^{n'} (t^{-1}_{j'_l} t_{j'_l +1})^{q'_l}
$,
where we set $t_{d+1} =t_1$.
Since $\pi(u) = \pi(v)$, it follows that
$\pi(u')=\pi(v')$, where
$
u'=
z^{\alpha+2q}  \prod_{l=1}^{m} x^{p_{l}}_{i_{l}} \prod_{l=1}^{n'} x^{q'_l}_{j'_l}$
and
$
v'=
z^{\alpha'+2 q'} \prod_{l=1}^{m'} x^{p'_l}_{i'_l}     \prod_{l=1}^{n} x^{q_{l}}_{j_{l}}.
$
Thus, $g= u'-v'$ belongs to $I_{A_d}$.
Since $g$ belongs to $K[z,X]$, $g$ belongs to the toric ideal
$I_B$, where
$B$ is the matrix
consisting of the first $d+1$ columns of $A_d$.
In addition, by virtue of $m,n, m',n' \leq k-1$,
we have 
$
| {\mathcal I} \cup {\mathcal J}' |  \leq  2(k -1) < d
$,
$
| {\mathcal I}' \cup {\mathcal J} | \leq  2(k -1) < d.
$
Thus, neither $u'$ nor $v'$ is divided by $x_1 \cdots x_d$.
Since $g \in I_B = \left< x_1 \cdots x_d - z^d \right>$, 
we have $g = 0$, that is, $u' = v'$.
Then, from ${\mathcal I} \cap {\mathcal J} = {\mathcal I}' \cap {\mathcal J}' =
{\mathcal I} \cap {\mathcal I}' = {\mathcal J} \cap {\mathcal J}' = \emptyset$,
we have
$
({\mathcal I} \cup {\mathcal J}') \cap ({\mathcal I}' \cup {\mathcal J}) = \emptyset
.$
Hence,
$\prod_{l=1}^{m} x^{p_{l}}_{i_{l}} \prod_{l=1}^{n'} x^{q'_l}_{j'_l}$ and 
$\prod_{l=1}^{m'} x^{p'_l}_{i'_l}     \prod_{l=1}^{n} x^{q_{l}}_{j_{l}}$ have
no common variables.
Since $u' = v'$, we have
$m=n=m'=n' =0$.
Hence,
$
u = z^{\alpha}
$
and
$
v = z^{\alpha'} 
$.
Since $f$ is a homogeneous binomial,
this is a contradicition.
Thus, ${\mathcal G}$ is a Gr\"obner basis of $I_{A_d}$.
It is trivial that ${\mathcal G}$ is reduced.
The case when $d$ is even is analyzed by a similar argument.
\end{proof}

\section{Ehrhart polynomials and roots}

For $0 \leq i \leq d$, 
let $r_d(i)$
denote the number of squarefree monomials in $K[X,Y]$
of degree $i$
that do not belong to the initial ideal ${\rm in}_{<} (I_{A_d})$
and
let $s_d(i+1)$
denote the number of squarefree monomials in $K[z,X,Y]$
of degree $(i+1)$
that are divided by $z$ and 
do not belong to ${\rm in}_{<} (I_{A_d})$.
For example, $r_d(0) = s_d(1) =1$ and $r_d(d) = 0$.

\begin{Lemma}
\label{standard}
For $0 \leq i  \leq  d-1$, we have
$r_d(i) = {d \choose i}
\sum_{\ell=1}^{d-i}
{i \choose k - \ell}$ and
$s_d(i+1) = r_d(i).$
In particular, 
$r_d(i) = {d \choose i} 2^i$
for
$0 \leq i \leq k-1.$
\end{Lemma}

\begin{proof}
Since the variable $z$ does not appear in the initial monomials 
of the binomials in Theorem \ref{oddGB},
$
u \notin {\rm in}_{<} (I_{A_d})
$
if and only if
$z \ u 
\notin {\rm in}_{<} (I_{A_d}) 
$
for any squarefree monomial $u \in K[X,Y]$.
Thus, 
$s_d(i+1) = r_d(i)$ for $0 \leq i  \leq  d-1$.
Suppose that $d$ is odd.
Then, from Theorem \ref{oddGB},  $r_d(i)$ is the number of monomials
$
\prod_{i \in {\mathcal I}} x_i
\prod_{j \in {\mathcal J}} y_j
$
where
${\mathcal I},{\mathcal J} \subset [d]$,
${\mathcal I} \cap {\mathcal J} = \emptyset$,
$|{\mathcal I} \cup {\mathcal J}| = i$
and $|{\mathcal I}|, |{\mathcal J}| \leq k-1$.
Since the number of subsets ${\mathcal I}$, ${\mathcal J} \subset [d]$ such that
${\mathcal I} \cap {\mathcal J} = \emptyset$,
$|{\mathcal I} \cup {\mathcal J}| = i$
and $|{\mathcal I}| = \lambda $ is
$
{d \choose \lambda , i - \lambda , d-i}
= 
{d \choose d-i}
{i \choose  \lambda}
= 
{d \choose i}
{i \choose  \lambda}
,$
it follows that
$
r_d(i) = 
\sum_{\ell=1}^{d-i}
{d \choose i}
{i \choose k - \ell}
=
{d \choose i}
\sum_{\ell=1}^{d-i}
{i \choose k - \ell}
$
for $0 \leq i  \leq  d-1$.
If $d$ is even, then the proof is similar.
\end{proof}

It is known \cite[Chapter 8]{Stu} that
${\rm in}_< (I_{A_d}) = \sqrt{{\rm in}_< (I_{A_d})}$ is the Stanley--Reisner ideal of
a regular unimodular triangulation $\Delta$ of ${\rm Conv}(A_d)$.
Thus, $r_d(i) + r_d(i+1)$ is the number of $i$-dimensional faces of $\Delta$.
From Lemma \ref{standard} and
\cite[Theorem 1.4]{Stanley},
 the Hilbert polynomial of $K[ z, X,Y] / I_{A_d}$ can be computed as follows:

\begin{Theorem}
\label{Ehrhart}
The Ehrhart polynomial of ${\rm Conv}(A_d)$ is
$
\sum_{i=0}^{d-1}
r_d(i)
{m \choose i}.
$
Moreover,
the normalized volume of  ${\rm Conv}(A_d)$ equals
$
k
{d \choose k}
.$
\end{Theorem}

Let $(h_0^{(d)},h_1^{(d)},\ldots,h_{d-1}^{(d)})$ be the $h$-vector of ${\rm Conv}(A_d)$.
Note that $h_{0}^{(d)} = 1 $.
Since ${\rm Conv}(A_d)$ is Gorenstein, we have $h_j^{(d)} = h_{d-1-j}^{(d)} $
for each $0 \leq j \leq d-1$.
Thus, it is enough to study $h_j^{(d)}$ for $1 \leq j \leq k-1$.

\begin{Theorem}
\label{hvector}
For $1 \leq j  \leq k-1$, we have
$$
h_{j}^{(d)} 
=
(-1)^j
\sum_{i=0}^{j}
(- 2)^i
{d \choose i}
{d - i -1 \choose j-i}
=
\left\{
\begin{array}{cc}
2^{d-1} & j = k-1 \mbox{ and } d \mbox{ is odd,}\\
h_{j}^{(d-1)} + h_{j-1}^{(d-1)} & \mbox{ otherwise.}
\end{array}
\right.
$$
\end{Theorem}

\begin{proof}
By
Lemma \ref{standard} and a well-known expression
(\cite[p. 58]{Stanley}), one can show
the first equality and that
$
h_{j}^{(d)} 
$ 
 is the coefficient of $u^j$ in the expansion of 
$
2^d (u+1)^d (u+2)^{-d+j}
$.
The second equality follows from 
this fact and the identity
given in \cite[p.148]{Stanley2}.
\end{proof}

\bigskip

Finally, we study the roots of the Ehrhart polynomial
 when $d$ is odd.
In this case, ${\rm Conv}(A_d)$ is a smooth Fano polytope of dimension $d-1$.
Since ${\rm Conv}(A_d)$ is a Gorenstein Fano polytope,
the roots of the Ehrhart polynomial
 are symmetrically distributed in the complex plane with respect to
the line ${\rm Re} (z) =  -1/2$.
Here, ${\rm Re}(z)$ is the real part of $z \in \CC$.
The following conjectures are given in \cite{BDDPS2005, five}:

\begin{Conjecture}[\cite{BDDPS2005}]
\label{conj:dstrip}
{\em
All roots $\alpha $ of Ehrhart polynomials of $D$ dimensional lattice polytopes
satisfy $-D \le {\rm Re} (\alpha )  \le D - 1$.
}
\end{Conjecture}

\begin{Conjecture}[\cite{five}]
\label{conj:five}
{\em
All roots $\alpha $ of Ehrhart polynomials of $D$ dimensional Gorenstein Fano polytopes
satisfy $-D/2 \le {\rm Re} (\alpha )  \le D/2 - 1$.
}
\end{Conjecture}

Using the software packages
{\tt Maple},
{\tt Mathematica},
and {\tt Maxima}, 
we computed
the largest real part of  roots of the Ehrhart polynomial of ${\rm Conv}(A_d)$:
\begin{table}[h]
\begin{center}
\begin{tabular}{|c|c|c|c|}
\hline
$d$ & $\dim ({\rm Conv}(A_d))$ & the largest real part & \\
\hline
35 & 34 & 16.35734046  &  a counterexample to
Conjecture \ref{conj:five}\\
\hline
125 & 124 & 123.5298262 & a counterexample to 
Conjecture \ref{conj:dstrip}
\\
\hline
127 & 126 & 126.5725840 &
greater than its dimension
\\
\hline
\end{tabular}
%
\end{center}
\end{table}

\begin{Remark}
{\rm
It was shown \cite{BrDe} that Conjecture \ref{conj:dstrip} is true for  $D \leq 5$.
Recently, a simplex (not a Fano polytope)
that does not satisfy
the condition ``${\rm Re} (\alpha )  \le D - 1$" in
Conjecture \ref{conj:dstrip} was presented in \cite{Higa2}.
Our polytope ${\rm Conv}(A_{125})$ is the first
example
satisfying neither 
``$-D \le {\rm Re} (\alpha ) $" 
nor
``${\rm Re} (\alpha )  \le D - 1$"
in
Conjecture \ref{conj:dstrip}.
}
\end{Remark}

\end{document}